\documentclass[12pt]{amsart}
\usepackage{amsmath}
\usepackage{amsthm}
\usepackage{amssymb}
\usepackage{amscd}
\usepackage{hyperref}
\usepackage{tikz}
\usetikzlibrary{matrix,arrows}

\usepackage{multirow}

\usepackage[all,cmtip]{xy}

\usepackage{footnote}
\makesavenoteenv{tabular}

\usepackage{comment}

\newcommand{\Tor}{\operatorname{Tor}}

\newcommand{\Hom}{\operatorname{Hom}}

\DeclareMathOperator{\lc}{H}

\DeclareMathOperator{\Minh}{Minh}

\newcommand{\Spec}{\operatorname{Spec}}
\newcommand{\length}{\ell}

\newcommand{\mf}{\mathfrak}
\newcommand{\frq}[1]{{#1}^{[q]}}
\newcommand{\ul}[1]{\underline{#1}}
\newcommand{\red}[1]{{#1}_{\text{red}}}

\DeclareMathOperator{\pd}{pd}

\DeclareMathOperator{\m}{\mathfrak{m}}
\DeclareMathOperator{\n}{\mathfrak{n}}
\DeclareMathOperator{\numg}{\nu}
\DeclareMathOperator{\Proj}{Proj}

\DeclareMathOperator{\rank}{rank}
\DeclareMathOperator{\lm}{c_{LM}}
\DeclareMathOperator{\im}{image}

\DeclareMathOperator{\eh}{e}
\DeclareMathOperator{\fsig}{s}
\DeclareMathOperator{\ehk}{e_{HK}}

\DeclareMathOperator{\Min}{Min}
\DeclareMathOperator{\Ann}{Ann}

\newtheorem{theorem}{Theorem}
\newtheorem{lemma}[theorem]{Lemma}
\newtheorem{proposition}[theorem]{Proposition}
\newtheorem{corollary}[theorem]{Corollary}

\newtheorem*{statement*}{Statement}
\newtheorem*{theorem*}{Theorem}
\newtheorem*{theoremA*}{Theorem A}
\newtheorem*{theoremB*}{Theorem B}
\newtheorem*{theoremC*}{Theorem C}
\newtheorem*{theoremD*}{Theorem D}
\newtheorem*{lemma*}{Lemma}
\newtheorem*{fact*}{Fact}

\theoremstyle{definition}

\newtheorem*{definition*}{Definition}
\newtheorem{example}[theorem]{Example}
\newtheorem*{example*}{Example}

\newtheorem{remark}[theorem]{Remark}
\newtheorem{question}[theorem]{Question}

\newcommand{\RomanNumeralCaps}[1]
    {\MakeUppercase{\romannumeral #1}}

\numberwithin{equation}{section}

\begin{document}

\title{Colength, multiplicity, and ideal closure operations \RomanNumeralCaps{2}}

\author{Linquan Ma}
\address{Department of Mathematics, Purdue University, West Lafayette, IN 47907 USA}
\email{ma326@purdue.edu}

\author{Pham Hung Quy}
\address{Department of Mathematics, FPT University, Hanoi, Vietnam}
\email{quyph@fe.edu.vn}

\author{Ilya Smirnov}
\address{BCAM-Basque Center for Applied Mathematics, Bilbao, Spain and Ikerbasque, Basque Foundation for Science, Bilbao, Spain}
\email{ismirnov@bcamath.org}



\maketitle

\begin{center}
{\textit{Dedicated to Professor Watanabe on the occasion of his 80th birthday}}
\end{center}

\begin{abstract}
Let $(R, \mf m)$ be a Noetherian local ring.
This paper concerns several extremal invariants arising from the
study of the relation between colength and (Hilbert--Samuel or Hilbert--Kunz) multiplicity
of an $\mf m$-primary ideal.
We introduce versions of these invariants by restricting to various closures and
``cross-pollinate'' the two multiplicity theories by asking for analogues invariants already established in one of the theories.

On the Hilbert--Samuel side, we prove that the analog of the St{\" u}ckrad--Vogel invariant (that is, the infimum of the ratio between the multiplicity and colength) for integrally closed $\m$-primary ideals is often $1$ under mild assumptions. We also compute the supremum and infimum of the relative drops of multiplicity for (integrally closed) $\m$-primary ideals. On the Hilbert--Kunz side, we study several analogs of the Lech--Mumford and St{\" u}ckrad--Vogel invariants.

\end{abstract}

\section{Introduction}

Let $(R, \mf m)$ be a Noetherian local ring of dimension $d$ and let $I$ be an $\mf m$-primary ideal. The Hilbert--Samuel multiplicity of $I$ is defined as
\[
\eh(I, R) = \lim_{n \to \infty} \frac{d! \length (R/I^n)}{n^{d}}.
\]
We will often abbreviate our notation and write $\eh(I)$ for $\eh(I, R)$ when $R$ is clear from the context and we will write $\eh(R)$ for $\eh(\mf m)$. The multiplicity is a positive integer that measures the singularities of $R$: when $\widehat{R}$ is unmixed\footnote{A Noetherian local ring $(S,\n)$ is unmixed if $S$ is equidimensional and satisfies Serre's condition $\text{S}_1$.}, $R$ is regular if and only if $\eh(R) = 1$, see \cite[Theorem 40.6]{Nagata}. It is closely connected to integral closure of ideals: when $\widehat{R}$ is equidimensional, for a pair of $\m$-primary ideals $I\subseteq J$, we have $\eh(I) = \eh(J)$
if and only if $J \subseteq \overline{I}$, see \cite{Rees1961}.

The starting point of this article is the following inequality of Lech \cite{LechMultiplicity} relating the multiplicity and colength of an $\m$-primary ideal:
\begin{theoremA*}[Lech's inequality]
\label{thm.Lech}
Let $(R,\m)$ be a Noetherian local ring of dimension $d$ and let $I$ be any $\m$-primary ideal of $R$. Then we have $$\eh(I)\leq d!\eh(R)\length(R/I).$$
\end{theoremA*}

The problem of improving Lech's inequality by replacing $\eh(R)$ with a smaller constant was raised and explored in \cite{HMQS}. This is partially motivated by \cite{MumfordStability}, where Mumford considered the quantity
\[
\sup_{\substack{\sqrt{I}=\m \\}} \left\{\frac{\eh(I)}{d!\length(R/I)} \right\}
\]
and showed its close connections with singularities on the compactification of the moduli spaces of smooth varieties constructed via Geometric Invariant Theory. Related study of the quantity above from an algebraic point of view appeared in \cite{HMQS} and \cite{MaSmirnovUniformLech}. A more detailed investigation will appear in \cite{MaSmirnovLechMumford}, where we call this quantity the {\it Lech--Mumford constant} of $R$ and denote it by $\lm(R)$. Note that Lech's inequality is simply saying that $\lm(R)\leq \eh(R)$, and by considering powers of an $\m$-primary ideal and taking limit, we have that $\lm(R)\geq 1$ essentially follows from the definition of multiplicity.



Now instead of studying the supremum of the ratio between multiplicity and colength, it is natural to consider the infimum:
\[
\inf\limits_{\sqrt{I}=\m} \left\{ \frac{\eh(I)}{\length(R/I)} \right\}.
\]
This quantity was studied by St\"{u}ckrad and Vogel \cite{StuckradVogel}, who conjectured that this infimum is positive precisely when $\widehat{R}$ is equidimensional. The conjecture was settled in full generality in the affirmative in \cite[Theorem 2.4]{KMQSY}. But inspired by the close relation between multiplicity and integrally closed ideals, it seems natural to also consider such infimum when we restrict ourselves to integrally closed ideals, that is,
\[
\inf\limits_{\sqrt{I}=\m, I=\overline{I}} \left\{ \frac{\eh(I)}{\length(R/I)} \right\}.
\]

While it is not known how to compute the St\"{u}ckrad--Vogel invariant
when the ring is not Cohen-Macaulay, our first main result shows that, somewhat surprisingly, the above infimum is $1$ under mild assumptions on $R$.
\begin{theoremB*}[=Theorem~\ref{thm: inf of integrally closed}]
Let $(R, \mf m)$ be a Noetherian local ring of dimension $d \geq 1$ such that $\widehat{R}$ is equidimensional.
Then we have
\[
\inf_{\sqrt{I} = \mf m, I=\overline{I}} \left\{\frac{\eh(I)}{\length(R/I)}\right\} = 1
\]
if and only if there exists a minimal prime $P$ of $\widehat{R}$ such that $\widehat{R}_P$ is a field.
\end{theoremB*}

Inspired by the relation between F-signature and the relative drop of Hilbert--Kunz multiplicity (see \cite{WatanabeYoshidaMin,PolstraTucker}), we next consider the infimum of the relative drops of Hilbert--Samuel multiplicities for integrally closed ideals
\[\inf_{\substack{I \subsetneq J, \sqrt{I} = \mf m \\ I = \overline{I}, J=\overline{J}}} \left \{\frac{\eh (I) - \eh (J)}{\length (R/I) - \length (R/J)}  \right\}.\]
We show that, similar to Theorem B above, this quantity is often $1$ under mild assumptions, see Theorem~\ref{thm: inf of relative drop of integrally closed}. On the other hand, we also show that, in stark contrast to Lech's inequality (i.e., $\lm(R)\leq \eh(R)$) and the behavior of Hilbert--Kunz multiplicity (see Proposition~\ref{prop: sup relative eHK}), the supremum of the relative drops of Hilbert--Samuel multiplicities for all (integrally closed) $\m$-primary ideals is often $\infty$, i.e., there is no {\it relative} Lech's inequality.

\begin{theoremC*}[=Proposition~\ref{prop: sup in dim at least two}]
Let $(R,\m)$ be a Noetherian local ring of dimension $d\geq 2$. Then we have
$$\sup_{\substack{I\subsetneq J,  \sqrt{I}=\m}} \left\{\frac{\eh (I) - \eh (J)}{\length (R/I) - \length (R/J)} \right\} = \infty.$$
If, in addition, $R/\m$ is algebraically closed, then
$$ \sup_{\substack{I\subsetneq J, \sqrt{I}=\m \\ I=\overline{I}, J=\overline{J}}} \left\{\frac{\eh (I) - \eh (J)}{\length (R/I) - \length (R/J)} \right\} = \infty.$$
\end{theoremC*}

We summarize our findings on the comparison between Hilbert--Samuel multiplicity and colength of $\m$-primary ideals in the following table. These results are proved in Section 3. To avoid triviality or tautology, we shall always assume that $\dim(R)\geq 1$ in our study of multiplicities since when $\dim(R)=0$, the multiplicity of any ideal is $\length(R)$ and the invariants we studied are not interesting or not meaningful.

\vspace{.4cm}

\begin{tabular}{ |p{3.7cm}||p{4.5cm}|p{5cm}| }
 \hline
 \multicolumn{3}{|c|}{Hilbert--Samuel multiplicity vs. colength} \\
 \hline
  & All ideals $I$, $J$ & Integrally closed ideals $I$, $J$\\
 \hline
 \hline
 $ \sup\limits_{\sqrt{I}=\m} \frac{\eh(I)}{d!\length(R/I)}$   & $\lm(R) \in [1, \eh(R)]$\footnote{In most cases, $\lm(R)<\eh(R)$ when $\dim(R)\geq 2$. In \cite{MaSmirnovUniformLech}, we completely characterized when $\lm(R)=\eh(R)$.}    & $\lm(R) \in [1, \eh(R)]$  \\
 \hline
  $\inf\limits_{\sqrt{I}=\m} \frac{\eh(I)}{\length(R/I)}$  & $1/n(R) \in [0,1]$\footnote{As we already mentioned, $\widehat{R}$ is equidimensional if and only if $n(R)<\infty$ (or equivalently, $1/n(R)>0$), see \cite[Theorem 2.4]{KMQSY}.} & often 1, Theorem~\ref{thm: inf of integrally closed} \\
 \hline
 $\sup\limits_{I\subsetneq J, \sqrt{I}=\m} \frac{\eh(I) - \eh(J)}{\length(R/I) - \length (R/J)}$ &  often $\infty$, Proposition~\ref{prop: sup in dim at least two}  & often $\infty$, Proposition~\ref{prop: sup in dim at least two} \\
 \hline
  $\inf\limits_{I \subsetneq J, \sqrt{I}=\m} \frac{\eh(I) - \eh(J)}{\length(R/I) - \length (R/J)}$ & 0 unless $R$ is a DVR & often 1, Theorem~\ref{thm: inf of relative drop of integrally closed} \\
 \hline
\end{tabular}

\vspace{.4cm}

In Section 4, we study the comparision between Hilbert--Kunz multiplicity and colength of $\m$-primary ideals for Noetherian local rings of prime characteristic $p>0$. As Hilbert--Kunz multiplicity has close relation with tight closure \cite[Theorem 8.17]{HochsterHuneke1}, besides considering various suprema and infima for integrally closed ideals, we also consider them when we restrict ourselves to tightly closed ideals.

Recall that if $(R, \mf m)$ is a Noetherian local ring of prime characteristic $p>0$ and dimension $d$, and $I$ is an $\mf m$-primary ideal, then the Hilbert--Kunz multiplicity of $I$ is defined as
\[
\ehk(I, R) = \lim_{n \to \infty} \frac{\length (R/I^{[p^e]})}{p^{ed}}.
\]
The limit above exists is a result of Monsky \cite{Monsky}. We will often abbreviate our notation and write $\ehk(I)$ for $\ehk(I, R)$ when $R$ is clear from the context and we will write $\ehk(R)$ for $\ehk(\mf m)$. Our results on the comparison between Hilbert--Kunz multiplicity and colength are summarized in the next table.

\vspace{.4cm}

\begin{tabular}{ |p{3.7cm}||p{3.4cm}|p{3.7cm}|p{4.1cm}| }
 \hline
 \multicolumn{4}{|c|}{Hilbert--Kunz multiplicity vs. colength} \\
 \hline
 & All ideals & Tightly closed ideals & Integrally closed ideals\\
 \hline
 \hline
 $\sup\limits_{\sqrt{I}=\m} \frac{\ehk(I)}{\length(R/I)}$  & $\ehk(R) \in [1, \eh(R)]$    & $\ehk(R)\in [1, \eh(R)]$ & $\ehk(R)\in [1, \eh(R)]$ \\
 \hline
  $\inf\limits_{\sqrt{I}=\m} \frac{\ehk(I)}{\length(R/I)}$   & ? see Remark~\ref{rmk: inf of eHK}  & ? see Remark~\ref{rmk: inf of eHK} & often 1, Proposition~\ref{prop: inf integrally closed Hilbert-Kunz} \\
 \hline
 $\sup\limits_{I\subsetneq J, \sqrt{I}=\m} \frac{\ehk(I) - \ehk(J)}{\length(R/I) - \length (R/J)}$ &   $\ehk(R)\in [1, \eh(R)]$, \, \ Proposition~\ref{prop: sup relative eHK} & ? see Corollary~\ref{cor: tight sup relative eHK}, \, \  Remark~\ref{rmk: sup eHK tc and ic} & ? see Remark~\ref{rmk: sup eHK tc and ic} \\
 \hline
  $\inf\limits_{I \subsetneq J, \sqrt{I}=\m} \frac{\ehk(I) - \ehk(J)}{\length(R/I) - \length (R/J)}$ & $\fsig(R)\in [0, 1]$ & ? see Question~\ref{question: infimum positive tc}, \, \ Corollary~\ref{cor: tight inf relative eHK sop} & often 1, Proposition~\ref{prop: inf integrally closed Hilbert-Kunz}\\
 \hline
\end{tabular}

\vspace{.4cm}

As one can see from this table, the first invariant, the analog of Lech--Mumford constant, turns out to be not so interesting as it always coincides with $\ehk(R)$. This is essentially well-known due to \cite[Lemma 4.2]{WatanabeYoshida} (i.e., the supremum is always achieved when $I=\m$). The same is true for the supremum of the relative drops of Hilbert--Kunz multiplicities if we consider all $\m$-primary ideals. Moreover, if we restrict ourselves to integrally closed ideals, then we show in Proposition~\ref{prop: inf integrally closed Hilbert-Kunz} that both infima are $1$ under mild assumptions. This follows almost immediately from our methods in the proofs of Theorem~\ref{thm: inf of integrally closed} and Theorem~\ref{thm: inf of relative drop of integrally closed}.

Beyond the discussion above, it seems that all other suprema and infima are interesting invariants that measure the singularities of $R$ in someway, and our knowledge about them is minimal. For example, the infimum of the relative drops of Hilbert--Kunz multiplicities of all ideals is known to be the F-signature of $R$, $\fsig(R)$, see \cite{PolstraTucker} (which partially motivates our study, as already mentioned above), and in general this invariant is known to be rather difficult to compute. It is known, however, that $\fsig(R)$ is positive if and only if $R$ is strongly F-regular so $\fsig(R)$ does not distinguish between non strongly F-regular singularities. On the other hand, if we restrict ourselves to tightly closed ideals in the relative drop of Hilbert--Kunz multiplicity, then we suspect that one always obtains a positive number, see Question~\ref{question: infimum positive tc}. In this direction, we prove the following result which can be viewed as a generalization of the main result of \cite{HochsterYao} (our method relies on \cite{HochsterYao} though).

\begin{theoremD*}[=Corollary~\ref{cor: tight inf relative eHK sop}]
Let $(R,\m)$ be an excellent, reduced and equidimensional local ring of prime characteristic $p>0$. Then we have
\[
\inf_{I=(\ul x)^*\subsetneq J=J^*} \left \{
\frac{\ehk (I) - \ehk (J)}{\length (R/I) - \length (R/J)}
\right \}
> 0,
\]
where $I=(\ul x)^*\subsetneq J=J^*$ means that $I$ is the tight closure of some system of parameters $(\ul x)$ of $R$ (i.e., we are allowed to vary the system of parameters).
\end{theoremD*}

For the other invariants (those marked with $?$ in the table), we collect some partial results and observations, and include some intriguing examples. In sum, our results on Hilbert--Kunz multiplicity and colength are less satisfactory, which partially justifies the philosophy that Hilbert--Kunz multiplicity is a more subtle invariant and is generally harder to study than Hilbert--Samuel multiplicity.

\subsection*{Acknowledgement}
We thank Keiichi Watanabe for asking what can be said about the relative Hilbert--Kunz multiplicity of integrally closed ideals. We are happy to report our answer in Proposition~\ref{prop: inf integrally closed Hilbert-Kunz}. The third author thanks Kriti Goel for conversations regarding this paper.

The first author is partially supported by NSF Grants DMS \#1901672, \#2302430, NSF FRG Grant \#1952366, and a fellowship from the Sloan Foundation. The second author is partially supported by a fund of Vietnam National Foundation for Science and Technology Development (NAFOSTED) under grant number 101.04-2023.08. The third author was supported by the Spanish Ministry of Science, Innovation and Universities through
the fellowship RYC2020-028976-I funded by MICIU/AEI/10.13039/501100011033 and El ESF ``ESF Investing in your future'',
the grant PID2021-125052GA-I00 funded by MICIU/AEI/10.13039/501100011033 and FEDER, UE,
and the grant EUR2023-143443 funded by MICIU/AEI/10.13039/501100011033 and the European Union NextGenerationEU/PRTR.

A part of this work was done during the third author's visit to Purdue University
and he thanks the Department of Mathematics for hospitality. This paper was finished while the second author visited the Vietnam Institute for Advanced Study in Mathematics (VIASM), he would like to thank the VIASM for the very kind support and hospitality.

\section{Preliminaries}

Throughout this manuscript, a ring $R$ will always be commutative and with multiplicative identity. Recall that an element $x\in R$ is integral over an ideal $I\subseteq R$ if it satisfies an equation of the form $x^n+ a_1x^{n-1}+\cdots + a_{n-1}x + a_n=0$, where $a_k\in I^k$. The set of all elements integral over $I$ forms an ideal and is denoted by $\overline{I}$, called the integral closure of $I$. An ideal $I\subseteq R$ is called integrally closed if $I=\overline{I}$.

We now list some basic properties of integral closure and integrally closed ideals. These properties will be freely used without further referencing.

\begin{enumerate}
\item The nilradical of $R$ is contained in the integral closure of any ideal $I$ (\cite[page 2]{SwansonHuneke}).
\item If $R\to S$ is any ring homomorphism and $I$ is an integrally closed ideal of $S$, then the pre-image (or contraction) of $I$ in $R$ is an integrally closed ideal in $R$ (\cite[page 3]{SwansonHuneke}).
\item An element $x\in R$ is integral over $I$ if and only if the image of $x$ in $R/\mf p$ is integral over $I(R/\mf p)$ for all minimal primes $\mf p$ (\cite[Proposition 1.1.5]{SwansonHuneke}).
\item If $R\to S$ is an integral extension, then $\overline{IS}\cap R=\overline{I}$ (\cite[Proposition 1.6.1]{SwansonHuneke}).
\item Let $(R,\m)$ be a Noetherian local ring. An $\m$-primary ideal $I$ is integrally closed in $R$ if and only if $I\widehat{R}$ is integrally closed in $\widehat{R}$ (\cite[Lemma 9.1.1]{SwansonHuneke}). In particular, there is a one-to-one correspondence between integrally closed $\m$-primary ideals in $R$ and $\widehat{R}$.
\item Let $(R,\m)$ be a Noetherian local ring with $R/\m$ infinite, then every $\m$-primary ideal $I$ is integral over an ideal $J$ generated by a system of parameters, such $J$ is called a minimal reduction of $I$ (\cite[Proposition 8.3.7 and Corollary 8.3.9]{SwansonHuneke}).
\item Let $(R,\m)$ be a Noetherian local ring and let $I, J$ be two $\m$-primary ideals that have the same integral closure. Then $\eh(I)=\eh(J)$ (\cite[Proposition 11.2.1]{SwansonHuneke}).
\end{enumerate}

We next assume that $R$ is a Noetherian ring of prime characteristic $p>0$. The tight closure of an ideal $I\subseteq R$, introduced by Hochster--Huneke \cite{HochsterHuneke1}, is defined as follows:
$$I^* := \{x\in R \,\ | \,\ \text{there exists $c\in R-\cup_{\mf p\in\Min(R)}\mf p$ such that $cx^{p^e}\in I^{[p^e]}$ for all $e\gg0$}\}.$$
An ideal $I\subseteq R$ is called tightly closed if $I=I^*$. In general, tight closure is always contained in the integral closure, i.e., $I^*\subseteq\overline{I}$ (\cite[Theorem 1.3]{Huneke}). Similar to integral closure, an element $x\in R$ is in the tight closure of $I$ if and only if the image of $x$ in $R/\mf p$ is in the tight closure of $I(R/\mf p)$ for all minimal primes $\mf p$ (\cite[Theorem 1.3]{Huneke}). If $(R,\m)$ is a Noetherian local ring and $I, J$ are two $\m$-primary ideals that have the same tigtht closure, then $\ehk(I)=\ehk(J)$ by \cite[Theorem 8.17]{HochsterHuneke1}. If $\dim(R)=d$, then $\eh(I)/d!\leq \ehk(I)\leq \eh(I)$, so in particular, $\ehk(I)=\eh(I)$ if $d=1$. We record the following lemma which was essentially proven in \cite[Corollary~5.8]{HochsterHuneke1}, but we slightly extend the level of generality.

\begin{lemma}
\label{lem: tc and ic in dimension one}
Let $(R,\m)$ be a one-dimensional Noetherian local ring of prime characteristic $p>0$. Then $I^*=\overline{I}$ for every ideal $I\subseteq R$.
\end{lemma}
\begin{proof}
If $R/\m$ is infinite, then $I$ is integral over a principal ideal $(x)$. Thus we know that
$$\overline{I}=\overline{(x)}=(x)^*\subseteq I^*\subseteq \overline{I}$$
where the second equality above follows from the tight closure Brian\c{c}on--Skoda theorem, see \cite[Theorem 2.7]{Huneke}.

Now we assume $k=R/\m$ is finite. Since every finite field extension $k\to k'$ is separable, we know that $k'\cong k[T]/f(T)$ for a monic polynomial $f(T)$. Let $F(T)\in R[T]$ be a monic polynomial that is a lift of $f(T)$ and consider $R':=R[T]/(F(T))$. Then $R\to R'$ is a finite free extension of Noetherian local rings with $\m R'$ the unique maximal ideal of $R'$,  and for sufficiently large $k'$, $IR'$ is integral over a principal ideal $(x)\subseteq R'$. The argument in the previous paragraph then shows that $\overline{IR'}=(IR')^*$. Thus we have that
$$\overline{I}=\overline{IR'}\cap R = (IR')^*\cap R\subseteq I^*$$
where the last inclusion follows from \cite[Theorem 5.22]{HochsterHunekeJAG}. Therefore $I^*=\overline{I}$ as wanted.
\end{proof}

We will often denote the target ring of the $e$-th iterates of the Frobenius endomorphism of $R$ as $F^e_*R$. $R$ is called F-finite if $F^e_*R$ is a finitely generated $R$-module for one (or equivalently, all) $e>0$. A Noetherian F-finite ring $R$ is called strongly F-regular if for every $c\in R-\cup_{\mf p\in\Min(R)}\mf p$, there exists $e>0$ such that the map $R\to F^e_*R$ sending $1$ to $F^e_*c$ splits as a map of $R$-modules. $R$ is called weakly F-regular if every ideal is tightly closed, and is called F-rational if every parameter ideal is tightly closed. In general, we have the following implications (see \cite{HochsterHuneke1},\cite{Huneke}):
\[
\text{regular} \Rightarrow \text{strongly F-regular} \Rightarrow \text{weakly F-regular} \Rightarrow \text{F-rational}.
\]

\subsection*{Flat local extensions of Noetherian local rings} We next record a few basic lemmas on faithfully flat extensions.
\begin{lemma}
\label{lem: gonflement}
Let $(R,\m)$ be a Noetherian local ring with $k=R/\m$. Then for any field extension $k\to k'$, there exists a flat local extension $(R,\m)\to (R',\m')$ of Noetherian local rings such that $\m'=\m R$ and $R'/\m'\cong k'$. Furthermore, if $R$ is excellent and $k$ is perfect, then $R\to R'$ is a regular morphism (i.e., all fibers of $R\to R'$ are geometrically regular).
\end{lemma}
\begin{proof}
The existence of $R'$ follows from gonflement, see \cite[Appendice 2]{BourbakiChapter8and9}. If $k$ is perfect, then the closed fiber of $R\to R'$ is a separable field extension and thus geometrically regular. Since $R$ is excellent, the formal fibers of $R$ are geometrically regular. Thus by Grothendieck's localization problem \cite{AndreGrothendieckLocalization}, we know that all the fibers of $R\to R'$ are geometrically regular.
\end{proof}

\begin{lemma}
\label{lem: equidimensional}
Let $(R,\m)\to (S,\n)$ be a flat local extension of Noetherian complete local rings such that $S/\m S$ is Gorenstein. Then we have
\begin{enumerate}
  \item If $R$ is equidimensional, then $S$ is equidimensional.
  \item If $R$ is unmixed, then $S$ is unmixed. 
\end{enumerate}
\end{lemma}
\begin{proof}
To prove $(1)$, note that since $R$ is equidimensional, it is enough to show $S/\mf pS$ is equidimensional for all minimal primes $\mf p$ of $R$. Hence we may assume $R$ is a Noetherian complete local domain and thus it suffices to prove $(2)$.

Now to prove $(2)$, first note that $R\to \Hom_R(\omega_R,\omega_R)$ is injective by \cite[(2.2) e)]{HochsterHunekeS2}, where $\omega_R$ denotes the canonical module of $R$. Since $S/\m S$ is Gorenstein, we know that $\omega_{S}\cong \omega_R\otimes_RS$ by \cite[Lemma 4.3]{SchwedeZhangBertini} (the equidimensional hypothesis there is not needed when $R$, $S$ are local). Thus $S\to \Hom_{S}(\omega_{S}, \omega_{S})$ is the base change of $R\to \Hom_R(\omega_R,\omega_R)$ along $R\to S$. It follows that $S\to \Hom_{S}(\omega_{S}, \omega_{S})$ is injective and thus $S$ is unmixed by \cite[(2.2) e)]{HochsterHunekeS2}.
\end{proof}

\subsection*{Chains of integrally closed ideals} The goal of this subsection is to revisit a handy result
of Watanabe \cite{WatanabeChain}, who showed that $\length (I/J)$ for integrally closed
ideals can be realized by a chain of integrally closed ideals under mild assumptions on $R$.
The original result \cite[Theorem~2.1]{WatanabeChain}
had a normality assumption, in \cite[Lemma~2.1]{HMQS} this assumption was relaxed to
irreducibility. However, a closer analysis of \cite[Proposition~3.1]{WatanabeChain} allows to remove this assumption and give a slightly different proof.

\begin{proposition}[Watanabe]\label{prop: improved Watanabe}
Let $(R, \m)$ be a Noetherian local ring
and $J \subsetneq I$ be
two $\m$-primary integrally closed ideals of $R$.
Set $k=R/\m$. Then there is a one-to-one correspondence between the set
\[
\{
I' \text { integrally closed ideal}, J \subseteq I' \subseteq I, \length (I/I') = 1
\}
\]
and the set of $k$-rational closed points of
\[
\Proj \left (\bigoplus_{n \geq 0} \frac{I^n}{JI^{n - 1} + \m I^n} \right).
\]
In particular, if $k$ is algebraically closed, there is always a chain of integrally closed ideals
\[
J = I_0 \subsetneq I_1 \subsetneq \cdots \subsetneq I_\ell = I
\]
such that $\length (I_{i}/I_{i-1}) = 1$ for all $i = 1, \ldots, \ell$.
\end{proposition}
\begin{proof}
An ideal $I' \subseteq I$ with $\length (I/I') = 1$ is given by
a subspace $V \subseteq I/\mf mI$ of codimension $1$.
If $I'$ is not integrally closed, then $I$ must be its integral closure,
hence $I'$ is integrally closed if and only if
$(\mf mI + V)I^{n-1} \subsetneq I^{n}$ for all $n$.
Going modulo $\mf mI^n$,  this condition is equivalent to saying that
the ideal generated by $V \subset I/\mf mI$ in the fiber cone $\mathcal{F}(I) := \oplus_{n \geq 0} I^n/\mf mI^n$
is not primary to the irrelevant ideal.
As $V$ has codimension $1$, in this case the ideal
$V\mathcal{F}(I)$ must be defining a
closed point in $\Proj \mathcal{F}(I)$. Hence we have a one-to-one correspondence
between integrally closed $I'$ and $k$-rational closed points in $\Proj \mathcal{F}(I)$.

Next we note that if $I'$ corresponds to $V$, then $J \subseteq I'$ if and only if $V$ contains the image
$(J + \mf mI)/\mf m I$, i.e., it corresponds to a homogeneous prime ideal in
\[
A := \bigoplus_{n \geq 0} \frac{I^n}{JI^{n - 1} + \m I^n}.
\]
Finally, we note that $I^n \neq JI^{n - 1}$ for all $n$ since $J$ is integrally closed. Hence, the irrelevant ideal of the standard graded algebra $A$ is not nilpotent, i.e., $\Proj A$ is not empty and thus must contain a $k$-rational closed point if $k$ is algebraically closed. This allows us to find $I_{\ell - 1}$ in the chain and we may continue by induction.
\end{proof}

We next note that the argument in \cite[Proposition~2.1]{CDHZ} holds more generally.

\begin{corollary}
\label{cor: integrally closed rank}
Let $(R, \mf m)$ be a Noetherian reduced local ring with $R/\m$ algebraically closed
and $J \subseteq I$ be $\m$-primary integrally closed ideals.
Suppose that $M$ is a finitely generated $R$-module that has a rank, i.e.,
there is an integer $r > 0$ such that  $\dim_{k(\mf p)} M_\mf p = r$
for all minimal primes $\mf p$ of $R$.
Then
\[
\length (IM/JM) \geq r \length (I/J).
\]
\end{corollary}
\begin{proof}
By Proposition~\ref{prop: improved Watanabe} the statement reduces to the case where $\length (I/J) = 1$, i.e., $I = (J, u)$ where $\m u \subseteq J$.
In this case, we let $N = JM :_M u$, so $M/N \cong IM/JM$
and the minimal number of generators of $M/N$ is $\length (IM/JM)$, which we denote by $\nu$. Now lift a system of minimal generators of $M/N$ to $M$ and let $N'$ be the submodule of $M$ generated by them. Since $u (M/N') \subseteq J (M/N')$,
the determinantal trick argument as in \cite[Proposition~2.1]{CDHZ} shows that there exists an element $a = u^n + j_1u^{n-1} + \cdots + j_n$ where $j_k\in J^k$ such that $a M \subseteq N'$.
Since $u \notin \overline{J}$ and because integral closure can be checked modulo minimal primes, $a \notin \mf p$ for some minimal prime $\mf p$ of $R$.
It follows that $M_\mf p = (N')_\mf p$. But the minimal number of generators
does not increase under localization so $\nu \geq r$ as required.
\end{proof}

In prime characteristic $p>0$ we obtain a consequence which will be crucial for treating
relative Hilbert--Kunz multiplicity. The result below is a generalization of \cite[Proposition 2.3]{CDHZ}, relaxing several of their assumptions.

\begin{corollary}
\label{cor: relative HK inequality}
Let $(R, \mf m)$ be a Noetherian local ring of prime characteristic $p > 0$ such that $\widehat{R}$ is equidimensional and $R/\m$ is perfect. If $J \subseteq I$ are integrally closed $\m$-primary ideals, then we have
\[
\ehk(J) - \ehk (I) \geq \length (I/J).
\]
\end{corollary}
\begin{proof}
We may assume $R=\widehat{R}$ is complete. By Lemma~\ref{lem: gonflement}, there exists a flat local extension $R\to R'$ with $R'/\m R'$ an algebraically closed field. Replacing $R'$ by $\widehat{R'}$ to assume $R'$ is complete, we know that $R\to R'$ is a regular morphism by Lemma~\ref{lem: gonflement}. In particular, $R\to R'$ is a normal morphism and thus $IR'$ is integrally closed whenever $I$ is integrally closed by \cite[Theorem 19.2.1]{SwansonHuneke}. Moreover, $R'$ is equidimensional by Lemma~\ref{lem: equidimensional} and it is clear that $\ehk(I)=\ehk(IR')$. Therefore, after replacing $R$ by $R'$, we may assume that $R$ is complete, equidimensional, and $R/\m$ is algebraically closed.

Now we note that $\length (I/J)$ is unaffected by passing to $\red{R}$ since $I, J$ are integrally closed.
Since Hilbert--Kunz multiplicity is additive in short exact sequences, we have
\begin{align*}
\ehk(J) - \ehk(I) & =
\ehk (J\red{R}) - \ehk(I\red{R}) + \ehk (J, N(R)) - \ehk (I, N(R)) \\
& \geq \ehk (J\red{R}) - \ehk(I\red{R}),
\end{align*}
where $N(R)$ denotes the nilradical of $R$. Thus it suffices to prove the assertion when $R$ is complete, reduced and equidimensional
with $R/\m$ algebraically closed. These conditions imply that $R$ is F-finite. Now by \cite[Proposition~2.3]{Kunz2}, the rank of $M = F^e_*R$ is $p^{e\dim(R)}$,
thus Corollary~\ref{cor: integrally closed rank} applies and the assertion follows after taking the limit as $e \to \infty$.
\end{proof}

We can use Corollary~\ref{cor: relative HK inequality} to give an alternative proof of Watanabe--Yoshida's
characterization of regular local rings \cite{WatanabeYoshida} via Hilbert--Kunz multiplicity.
The key step in their argument is to find an $\mf m$-primary ideal $I \subseteq \m^{[p]}$
such that $\ehk (I) \geq \length (R/I)$. The original proof took $I$ to be a parameter ideal, which required a significant amount of work to show that $R$ is Cohen--Macaulay. A simpler argument of Huneke--Yao \cite{HunekeYao} constructs such an ideal by using that regular locus is large by induction on dimension. We will use instead integrally closed ideals.

\begin{theorem}[Watanabe--Yoshida]
\label{thm: Watanabe--Yoshida}
Let $(R, \mf m)$ be a Noetherian local ring of prime characteristic $p > 0$ such that $\widehat{R}$ is unmixed. If $\ehk (R) = 1$, then $R$ is regular.
\end{theorem}
\begin{proof}
We may assume that $R=\widehat{R}$ is complete. Let $k\cong R/\m$ be a coefficient field of $R$. Applying Lemma~\ref{lem: equidimensional}
to the map $R \to R \widehat{\otimes}_k \overline{k}$, we reduce the problem
to the case where $R/\m$ is algebraically closed.
Now, we note that
\[
1 = \ehk (R) = \sum_{\mf p \in \Minh (R)} \length (R_\mf p) \ehk (R/\mf p)
\geq \sum_{\mf p \in \Minh (R)} \length (R_\mf p)
\]
where the sum varies through all minimal primes $\mf p$ of $R$ such that $\dim (R/\mf p) = \dim (R)$. By the unmixedness assumption these are all the \emph{associated primes} of $R$, it follows that $R$ is a domain.

Since $R$ is a Noetherian complete local domain, there is a constant $c$ such that
$\overline{\mf m^{c+n}} \subseteq \mf m^{n}$ by \cite[Proposition~5.3.4]{SwansonHuneke}. Thus we can find an integrally closed ideal $I := \overline{\mf m^n} \subseteq \m^{[p]}$.
Now we have
\[
 \length (\mf m^{[p]}/I)  + p^{\dim(R)} = \length (\mf m^{[p]}/I) \ehk (R) + \ehk (\mf m^{[p]}) \geq \ehk (I) \geq \length (R/I) = \length (\mf m^{[p]}/I) + \length (R/\mf m^{[p]}),
\]
where the first inequality follows from \cite[Lemma 4.2]{WatanabeYoshida} and the second inequality follows from Corollary~\ref{cor: relative HK inequality}. Thus we obtain that $p^{\dim(R)} \geq \length (R/\mf m^{[p]})$. This implies that $R$ is regular by the celebrated result of Kunz \cite{Kunz1}.
\end{proof}

Finally, we also give an improvement of \cite[Corollary~2.4]{CDHZ}.

\begin{corollary}
\label{cor: HK integral closed and singular}
Let $(R, \mf m)$ be a Noetherian local ring of prime characteristic $p > 0$ such that $\widehat{R}$ is equidimensional and $R/\m$ is perfect. If $I$ is $\m$-primary and integrally closed
then $\ehk (I) \geq \length (R/I)$. Moreover, if $\widehat{R}$ is unmixed and equality holds for some (integrally closed) $I$, then $R$ is regular.
\end{corollary}
\begin{proof}
By Corollary~\ref{cor: relative HK inequality},
$\ehk (I) \geq \length (R/I) + \ehk (\mf m) - 1 \geq \length (R/I)$.
If equality holds, then $\ehk (\mf m) = 1$,
so $R$ is regular by Theorem~\ref{thm: Watanabe--Yoshida}.
\end{proof}

\section{Multiplicity and colength}

\subsection{Infima} In this section we prove our main results on the comparison between multiplicity and colength. We start with the following lemma which essentially is a consequence of Flenner's local Bertini theorem \cite[(3.1) Satz]{Flenner} (see also \cite{Trivedi, TrivediE} for correction in the mixed characteristic case).
\begin{lemma}
\label{lem.Flenner}
Let $(R, \mf m)$ be a Noetherian complete local ring of dimension $d \geq 2$. Suppose there exists a minimal prime $P$ of $R$ such that $\dim(R/P)=d$ and $R_P$ is a field. Then there exists an element $x \in \m$ such that $\dim(R/(x))=d-1$ and $R/(x)$ has a minimal prime $P'$ with $\dim(R/P')=d-1$ and $(R/(x))_{P'}$ a field.
\end{lemma}
\begin{proof}
Since $R$ is complete, we know that the singular locus of $\Spec(R)$ is closed. Let $V(I)$ be the singular locus of $\Spec(R)$ and our assumption implies that $P\nsupseteq I$. It follows that $P+I$ has height at least one. Let $\{Q_1,\dots,Q_n\}$ be the height one minimal primes of $P+I$ (if the height of $P+I$ is at least two, then this is simply the empty set).
By \cite[(3.1) Satz]{Flenner}, there exists an element $x\in \m$ that avoids $Q_1,\dots,Q_n$ and all minimal primes of $R$ such that
\begin{equation}
\label{eqn: flenner}
\text{Reg}(R/(x)) \cap \Spec^\circ(R) \supseteq \text{Reg}(R) \cap V(x) \cap \Spec^\circ(R)
\end{equation}
where $\Spec^\circ(R)$ denotes the punctured spectrum of $R$ and $\text{Reg}(R)$ denotes the regular locus of $\Spec(R)$. Since $\dim(R)=\dim(R/P)=d$ and by our choice $x$ is not inside any minimal prime of $R$, we know that $$\dim(R/(x))=\dim(R/P+(x))=d-1.$$
Thus there exists a minimal prime $P'$ of $P+(x)$ such that $\dim(R/P')=d-1$. Note that $P'$ is also a minimal prime of $(x)$ since $\dim(R/(x))=d-1$. We next claim that $P'$ is contained in the right hand side of (\ref{eqn: flenner}): it is clear that $P'$ is contained in $V(x)\cap\Spec^\circ(R)$, if $R_{P'}$ is not regular, then $P'\supseteq I$ and hence $P'\supseteq I+(x)$, but then since $P'$ has height one it must coincide with one of $Q_1,\dots, Q_n$ which is a contradiction since $x$ avoids these primes. Thus $P'$ is contained in the right hand side of (\ref{eqn: flenner}) and hence also the left hand side of (\ref{eqn: flenner}). In particular, $(R/(x))_{P'}$ is regular, and thus a field.
\end{proof}

We next record the following simple lemma for one-dimensional rings.

\begin{lemma}
\label{lem: one dimensional complete domain}
Let $(R,\m)$ be a Noetherian complete local domain of dimension one. Then there exist an integer $k$ and a constant $C>0$ such that there exists a sequence of integrally closed ideals $\{I_n\}_n$ satisfying $\m^{kn}\subseteq I_n\subseteq \m^n$ for all $n\geq 1$ and $$\eh (I_n) \leq \left (1 + \frac{C}{n} \right ) \length (R/I_n)$$ for all $n\gg0$.
\end{lemma}
\begin{proof}
By a result of Rees \cite{Rees},
there exists an integer $k$ such that $\overline{\mf m^{kn}} =  (\overline{\mf m^{k}})^n$ for all $n \geq 1$.
It follows from Hilbert--Samuel theory that
there exists a constant $D$ such that for all $n \gg 0$ we have
\[
\length(R/\overline{\mf m^{kn}}) =
\length(R/(\overline{\mf m^{k}})^n)
= n \eh(\overline{\mf m^{k}}) - D
= \eh(\overline{\mf m^{kn}}) -D.
\]
Now it is easy to see that the result follows by setting $I_n=\overline{\m^{kn}}$.
\end{proof}

The following is our key lemma.
\begin{lemma}
\label{lem: main approx}
Let $(R, \mf m)$ be a Noetherian local ring of dimension $d \geq 1$ that has a minimal prime $P$ of $\widehat{R}$ such that $\dim(\widehat{R}/P)=d$ and $\widehat{R}_P$ is a field.
Then there exist a constant $C>0$ and a sequence of $\m$-primary integrally closed ideals $\{J_n\}_n$ such that
$$\eh (J_n) \leq \left (1 + \frac{C}{n} \right ) \length (R/J_n)$$
for all $n\gg0$ and that $\length(R/J_n)\to\infty$. In particular, we have
$$\inf_{\sqrt{I} = \mf m, I=\overline{I}} \left\{\frac{\eh(I)}{\length(R/I)}\right\} \leq 1.$$
\end{lemma}
\begin{proof}
We may assume that $R=\widehat{R}$ is complete. We first suppose $d=1$, so that $R/P$ is a complete local domain of dimension one. By Lemma \ref{lem: one dimensional complete domain} there exist an integer $k$, a constant $C'>0$, and a sequence
of integrally closed ideals $\{I_n\}_n$ of $R/P$ such that $\mf m^{kn}(R/P) \subseteq I_n \subseteq \mf m^n(R/P)$ and such that
$$\eh(I_n, R/P) \leq (1 + C'/n) \length ((R/P)/I_n(R/P))$$ for all $n\gg0$. Let $J_n$ be the pre-image of $I_n$ in $R$. Then $J_n$ is $\m$-primary and integrally closed in $R$.
By the additivity formula for multiplicities, we have
\[
\eh(J_n) = \sum_{\mf p \in \Min(R)} \length (R_\mf p) \eh (J_n, R/\mf p)
= \eh (I_n, R/P) + \sum_{\mf p \neq P} \length (R_\mf p) \eh (J_n, R/\mf p).
\]
By Lech's inequality, $\length (R_\mf p) \eh (J_n, R/\mf p) \leq
\length (R_\mf p)\eh(R/\mf p) \length (R/(\mf p + J_n)) \leq \eh(R) \length (R/(\mf p + J_n))$. It follows that
\[
\eh(J_n) \leq \eh (I_n, R/P) + \eh(R) \sum_{\mf p \neq P} \length (R/(\mf p + J_n)).
\]
Now, if $\mf p \neq P$ then $\dim (R/(\mf p + P))=0$. Since $J_n\subseteq \m^n+P$, we have
\[
\frac{\length (R/(\mf p + J_n))}{\length (R/J_n)}
\leq \frac{\length (R/(\mf p + P))}{\length (R/(\mf m^{n}+P))} \leq \frac{C(\mf p)}{n}
\]
where $C(\mf p)>0$ is a constant depending on $\mf p$. Therefore we have
\[
\frac{\eh(J_n)}{\length (R/J_n)}
= \frac{\eh (I_n, R/P) + \eh(R) \sum_{\mf p \neq P} \length (R/(\mf p + J_n))}{\length (R/J_n)}
\leq 1+\frac{C}{n}.
\]
for all $n\gg0$ for some constant $C>0$ (depending on $C'$, $C(\mf p)$). It is clear that $\length(R/J_n)\to\infty$. This completes the proof when $d=1$.

Now we assume $d \geq 2$ and we use induction. By Lemma~\ref{lem.Flenner}, there exists an element $x \in \mf m$ such that
$S = R/(x)$ has dimension $d-1$ and $S$ satisfies the same assumptions as $R$. By induction, we have a sequence of integrally closed $\m$-primary ideals $\{I_n\}_n$ in $S$
such that $\eh (I_n) \leq (1 + C/n) \length (S/I_n)$.
Let $J_n$ be the pre-image of $I_n$ in $R$. Then $J_n$ is $\m$-primary and integrally closed.
Therefore, as $x \in J_n$, we have
\[
\frac{\eh(J_n)}{\length (R/J_n)} \leq \frac{\eh(J_n, R/(x))}{\length (R/J_n)} =
\frac{\eh(I_n)}{\length (S/I_n)} \leq 1 + \frac{C}{n}.
\]
This completes the proof.
\end{proof}

Now we are ready to state and prove our main result on the ratio between multiplicity and colength for integrally closed ideals.

\begin{theorem}
\label{thm: inf of integrally closed}
Let $(R, \mf m)$ be a Noetherian local ring of dimension $d \geq 1$ such that $\widehat{R}$ is equidimensional.
Then we have
\[
\inf_{\sqrt{I} = \mf m, I=\overline{I}} \left\{\frac{\eh(I)}{\length(R/I)}\right\} = 1
\]
if and only if there exists a minimal prime $P$ of $\widehat{R}$ such that $\widehat{R}_P$ is a field.

Moreover, if $\widehat{R}$ is unmixed, then
the infimum is achieved if and only if $R$ is regular.
\end{theorem}
\begin{proof}
First of all, \cite[Corollary~12]{MQS}
asserts that $\eh(I) \geq \length(R/I)$ for any integrally closed $\m$-primary ideal and that equality holds for some $I$ implies $R$ is regular (when $\widehat{R}$ is unmixed). This shows the second assertion, and combined with Lemma \ref{lem: main approx}, this proves the ``if" direction of the first assertion.

It remains to prove the ``only if" direction of the first assertion. We may assume $R=\widehat{R}$ is complete. Suppose there is no minimal prime $P$ such that $R_P$ is a field. Then the nilradical $N(R)$ of $R$ is supported on the entire $\Spec (R)$, i.e., $N(R)_\mf p \neq 0$ for all $\mf p\in\Spec(R)$.
Hence, we may apply \cite[Theorem~4.4]{KMQSY} to
find a constant $C_1 > 0$
such that $$\length (N(R)/IN(R)) \geq C_1 \length (R/I)$$ for every $\mf m$-primary ideal $I$. Furthermore, by \cite[Theorem~2.4]{KMQSY}, there exists another constant $C_2>0$ such that $$\eh(I, N(R)) \geq C_2\length (N(R)/IN(R)).$$
By the additivity property of multiplicity, we have
\[
\frac{\eh(I)}{\length(R/I)}
= \frac{\eh(I, \red{R}) + \eh(I, N(R))}{\length(R/I)}
\geq \frac{\eh(I, \red{R}) + C_2\length (N(R)/IN(R))}{\length(R/I)}
\geq \frac{\eh(I, \red{R})}{\length(R/I)} + C_1C_2.
\]
Since $N(R)$ is contained in any integrally closed ideal,
we may apply \cite[Corollary~12]{MQS} to $\red{R}$ to see that $\eh(I, \red{R}) \geq \length(\red{R}/I\red{R}) =\length(R/I)$. Therefore we have
$$\eh(I)/\length(R/I) \geq (1 + C_1C_2) > 1$$
for every integrally closed $\mf m$-primary ideal $I$.
\end{proof}

Inspired by the relation between F-signature and the relative drop of the Hilbert--Kunz multiplicity (see \cite{WatanabeYoshidaMin,PolstraTucker}). We next consider the infimum on the relative drop of the Hilbert--Samuel multiplicity for integrally closed ideals.

\begin{theorem}
\label{thm: inf of relative drop of integrally closed}
Let $(R, \mf m)$ be a Noetherian local ring of dimension $d \geq 1$. Then
\begin{align*}
&\inf_{\substack{I \subsetneq J, \sqrt{I} = \mf m \\ I=\overline{I}}} \left \{\frac{\eh (I) - \eh (J)}{\length (R/I) - \length (R/J)} \right\} =
\inf_{\substack{I \subsetneq J, \sqrt{I} = \mf m \\ I = \overline{I}, J=\overline{J}}} \left \{\frac{\eh (I) - \eh (J)}{\length (R/I) - \length (R/J)}  \right\}.
\end{align*}
Furthermore, we have
\begin{enumerate}
  \item If $\widehat{R}$ is equidimensional and $R/\m$ is perfect, then the infimum above is $\geq 1$.
  \item If there exists a minimal prime $P$ of $\widehat{R}$ such that $\dim(\widehat{R}/P)=d$ and $\widehat{R}_P$ is a field, then the infimum above is $\leq 1$.
\end{enumerate}
\end{theorem}
\begin{proof}
Clearly, for any $I\subsetneq J$ we have
\[
\frac{\eh (I) - \eh (J)}{\length (R/I) - \length (R/J)}
\geq \frac{\eh (I) - \eh (\overline{J})}{\length (R/I) - \length (R/\overline{J})}
\]
and, since the second infimum is taken over a smaller set,
the two infima are always equal. It remains to prove $(1)$ and $(2)$.

\begin{proof}[Proof of $(1)$] First of all, if $R/\m$ is algebraically closed, then by Proposition~\ref{prop: improved Watanabe}, for all $\m$-primary integrally closed ideals $I\subsetneq J$, we have
$$\eh (I) - \eh(J) = \sum_{i = 1}^{\length (J/I)} \left( \eh(I_i) - \eh (I_{i-1}) \right)
\geq \length (J/I)=\length(R/I)-\length(R/J). $$
Here the only inequality used the fact that $\eh(I_i)\geq \eh(I_{i-1})+1$ since both ideals are integrally closed and $\widehat{R}$ is equidimensional. Thus the infimum in question is $\geq 1$.

Now we assume $k=R/\m$ is perfect. We may assume $R=\widehat{R}$ is complete. By Lemma~\ref{lem: gonflement}, there exists a flat local extension $R\to R'$ with $R'/\m R'$ an algebraically closed field. Replacing $R'$ by $\widehat{R'}$ to assume $R'$ is complete, we know that $R\to R'$ is a regular morphism by Lemma~\ref{lem: gonflement}. In particular, $R\to R'$ is a normal morphism and thus $IR'$ is integrally closed whenever $I$ is integrally closed by \cite[Theorem 19.2.1]{SwansonHuneke}. Moreover, $R'$ is equidimensional by Lemma~\ref{lem: equidimensional} and it is clear that $\eh(I)=\eh(IR')$ and $\length(R/I)=\length(R'/IR')$.
It follows from these discussions that
$$\inf_{\substack{I \subsetneq J, \sqrt{I} = \mf m \\ I = \overline{I}, J=\overline{J}}} \left \{\frac{\eh (I) - \eh (J)}{\length (R/I) - \length (R/J)}  \right\} \geq \inf_{\substack{I' \subsetneq J', \sqrt{I'} = \mf m R' \\ I' = \overline{I'}, J'=\overline{J'}}} \left \{\frac{\eh (I') - \eh (J')}{\length (R'/I') - \length (R'/J')} \right\}. $$
Since $R'$ is (complete and) equidimensional with $R'/\m R'$ algebraically closed, we know that the latter infimum is $\geq 1$. The result follows.
\end{proof}

\begin{proof}[Proof of $(2)$]
By Lemma \ref{lem: main approx}, there exists a constant $C>0$ and a sequence of
integrally closed $\m$-primary ideals $\{I_n\}_n$ such that $\eh (I_n)\leq (1 + C/n) \length (R/I_n)$
and $\length (R/I_n)\to\infty$ as $n\to\infty$. It follows that
\[
\inf_{\substack{I \subsetneq J, \sqrt{I} = \mf m \\ I = \overline{I}, J=\overline{J}}} \left \{\frac{\eh (I) - \eh (J)}{\length (R/I) - \length (R/J)}  \right\} \leq \inf_n \frac{\eh (I_n) - \eh (\mf m)}{\length (R/I_n) - \length (R/\mf m)}
\leq \lim_{n \to \infty }
\frac{(1 + C/n) \length (R/I_n) - \eh (\mf m)}{\length (R/I_n) - \length (R/\mf m)}
= 1
\]
since $\length(R/I_n)\to\infty$. This completes the proof of $(2)$.
\end{proof}
\end{proof}

In Theorem~\ref{thm: inf of integrally closed} (and Theorem~\ref{thm: inf of relative drop of integrally closed} (1)), we assumed $\widehat{R}$ is equidimensional. We next point out that, when $\widehat{R}$ is not equidimensional, both infima studied above are $0$. This should be viewed as a generalization of \cite[Theorem 1 (1)]{StuckradVogel} (essentially, the result is proved there when the infimum is ranged over all $\m$-primary ideals).

\begin{proposition}
\label{prop: non equidimensional}
Let $(R, \mf m)$ be a Noetherian local ring of dimension $d \geq 1$ such that $\widehat{R}$ is not equidimensional. Then
$$\inf_{\sqrt{I} = \mf m, I=\overline{I}} \left\{\frac{\eh(I)}{\length(R/I)}\right\}= \inf_{\substack{I \subsetneq J, \sqrt{I} = \mf m \\ I = \overline{I}, J=\overline{J}}} \left \{\frac{\eh (I) - \eh (J)}{\length (R/I) - \length (R/J)} \right\}=0.$$
\end{proposition}
\begin{proof}
We may assume $R=\widehat{R}$ is complete. Let $\mf p_1,\dots, \mf p_t$ be all the minimal primes of $R$ such that $\dim(R/\mf p_i)=d$. We first claim that there exists $Q\in\Spec(R)\setminus\{\m\}$ such that $Q(R/\mf p_i)$ is $\m(R/\mf p_i)$-primary for every $i$. To see this, note that since $R$ is not equidimensional, there exists a minimal prime $P$ of $R$ such that $\dim(R/P)<d$. We know that $\mf q:= \cap_i\mf p_i$ is not contained in $P$. Let $x_1,\dots, x_s$ be a system of parameters of $R/(P+\mf q)$ and let $Q$ be a minimal prime of $(x_1,\dots,x_s)(R/P)$. Then we know that $Q\in\Spec(R)\setminus\{\m\}$, $Q$ contains $P$, and that $Q(R/\mf p_i)$ is $\m(R/\mf p_i)$-primary (as $Q+\mf q$ is $\m$-primary).

Now we let $I_n$ be the pre-image of $\overline{\m^n(R/Q)}$ in $R$. Then we know that $I_n$ is $\m$-primary and integrally closed, and that
$$\length(R/I_n)=\length\left((R/Q)/{\overline{\m^n(R/Q)}}\right)\to\infty.$$
On the other hand, by the additivity formula for multiplicities, we know that
$$\eh(I_n)= \sum_i\eh(I_n, R/\mf p_i)\length(R_{\mf p_i})\leq \sum_i\eh(Q(R/\mf p_i), R/\mf p_i)\length(R_{\mf p_i}) =C$$
where $C$ is a constant independent of $n$. Putting these together we see that
$$\lim_{n\to\infty} \frac{\eh(I_n)}{\length(R/I_n)} =\lim_{n\to\infty} \frac{\eh (I_n) - \eh (\mf m)}{\length (R/I_n) - \length (R/\mf m)} =0.$$
This clearly implies the desired conclusion.
\end{proof}

\begin{remark}
One might also consider $$\inf_{\substack{I \subsetneq J, \sqrt{I} = \mf m}} \left \{\frac{\eh (I) - \eh (J)}{\length (R/I) - \length (R/J)} \right\}.$$
But it is clear that this infimum is $0$ as long as there exists an $\m$-primary ideal $I$ that is not integrally closed (for then setting $J=\overline{I}$, the numerator above is $0$). Therefore this invariant is $0$ unless $R$ is a DVR, in which case every $\m$-primary ideal is integrally closed with multiplicity equals to colength, in which case this invariant is $1$.
\end{remark}

\subsection{Suprema} In this subsection, we study the suprema of the ratio between multiplicity and colength, and of the relative drops of Hilbert--Samuel multiplicities. The answers are completely different when the dimension of the ring is one or larger than one. We first handle the one-dimensional case. We begin with a simple and well-known lemma.

\begin{lemma}
\label{lem: dim one}
Let $(R,\m)$ be a Noetherian local ring of dimension one. Then for any two $\m$-primary ideals $\mathfrak{a},\mathfrak{b}$ we have
$$\eh(\mathfrak{a}\mathfrak{b})=\eh(\mathfrak{a})+\eh(\mathfrak{b}).$$
\end{lemma}
\begin{proof}
We may assume that $R/\m$ is infinite. We can replace $\mathfrak{a},\mathfrak{b}$ by their minimal reductions $(a), (b)$, respectively. Since $\eh(a, R)= \chi(a, R)$, the result follows from additivity of Euler characteristic (for example see \cite[Lemma 2.1]{MaLimUlrich}).
\end{proof}

\begin{proposition}
\label{prop: sup in dim one}
Let $(R,\m)$ be a Noetherian local ring of dimension one. Then we have
$$\sup_{\substack{I\subsetneq J,  \sqrt{I}=\m}} \left\{\frac{\eh (I) - \eh (J)}{\length (R/I) - \length (R/J)} \right\} = \eh(R).$$
\end{proposition}
\begin{proof}
To see the supremum is $\leq \eh(R)$, it is enough to show that when $\ell(J/I)=1$ we have $\eh(I)-\eh(J)\leq \eh(R)$. But if $\length(J/I)=1$, then $\m J \subseteq I$ and thus $\eh(I)\leq \eh(\m J) = \eh(R)+ \eh(J)$ by Lemma \ref{lem: dim one}. To see the supremum is equal to $\eh(R)$, let $x\in \m$ be a parameter element and consider $I=\m(x)\subseteq J=(x)$. By Lemma \ref{lem: dim one}, we know that $\eh(I)=\eh(R)+ \eh(J)$ and $\length(J/I)=1$. Thus $\eh (I) - \eh (J)= \eh(R)\cdot (\length(R/I)-\length(R/J))$ in this case.
\end{proof}

In dimension one, the supremum in Proposition~\ref{prop: sup in dim one} might be less than $\eh(R)$ if we restrict ourselves to integrally closed ideals.

\begin{example}
\label{example: sup of integrally closed dim one}
Consider $R = k[[x,y]]/(xy)$. We claim that
\[
\sup_{\substack{I\subsetneq J, \sqrt{I}=\m \\ I=\overline{I}, J=\overline{J}}} \left\{\frac{\eh (I) - \eh (J)}{\length (R/I) - \length (R/J)} \right\} = 1 <2= \eh(R).
\]

First of all, any integrally closed $\m$-primary ideal $I$ must have two generators. To see this, note that since $I$ is $\m$-full (\cite[Theorem~3]{WatanabeJ})
we know that $\numg (I) \geq \numg(\m) = 2$, but, on the other hand,
since $x +y$ is a minimal reduction of $\m$ so
$$\numg(I) \leq \length (I/(x+y)I) = \eh ((x+y), I) = \eh ((x + y), R) = 2.$$

Next, we prove that every integrally closed ideal of $R$ has
the form $(x^a, y^b)$ for some positive integers $a, b$.
Namely, an $\m$-primary ideal must contain some power $(x + y)^n = x^n + y^n$,
hence it contains $x^{n + 1} = x (x + y)^n$ and $y^{n + 1} = y (x + y)^n$.
Thus, there are minimal positive integers $a, b$ such that $x^a, y^b \in I$.
By contradiction, suppose that $I \neq (x^a, y^b)$. Then $I$ contains an element of the form
\[
f = \sum_{i = 1}^{a - 1} \alpha_i x^i + \sum_{j = 1}^{b - 1} \beta_j y^j
\]
with $\alpha_i, \beta_j \in k$.
Suppose that some $\alpha_i \neq 0$. Multiplying by $x$ we obtain that
$0 \neq xf = \sum_{i = 0}^{a - 1} \alpha_i x^{i + 1} \in I$.
Let $n = \min \{i \mid \alpha_i \neq 0\}$, then $xf = x^{n + 1} \cdot \text{unit}$,
so $n + 1 = a$ by the minimality of $a$.
The minimality also implies that some $\beta_j \neq 0$, so an identical argument for $y$  shows that $x^{a - 1} + \text{unit} \cdot y^{b-1} \in I$.
However, both $x^{a - 1}$ and $y^{b-1}$ are integral over this element, since we can check this modulo minimal primes. Thus we get a contradiction with the definition of $a$ and $b$.

Finally, it is easy to see that $\eh ((x^a, y^b), R) = a + b$ and $\length (R/(x^a, y^b)) = a + b - 1$.
Thus each term in the supremum above is $1$, so the supremum is $1$ as wanted.
\end{example}

\begin{example}\label{example: sup of integrally closed numerical semigroups}
Using numerical semigroup rings we may provide examples where the invariant
takes different values between $1$ and the multiplicity.

Since the integral closure of a numerical semigroup ring $R = k[S]$ is the polynomial ring $k[T]$, we can classify the integrally closed ideals as
\[
I_n = \left \{(T^n) \cap R \mid 1 \neq n \in S \right \}.
\]
The multiplicity of $I_n$ is $n$ while $\length (R/I_n) = \# \{s \in S \mid s < n\}$.
Therefore, once $n$ passes the conductor, we have the formula $\length (R/I_n) =  n - \# (\mathbb{N} \setminus S)$.
Thus the relative Hilbert--Samuel multiplicity of two integrally closed ideals contained in the conductor is always $1$. This observation and some simple computations show that
\begin{enumerate}
  \item For the semigroup ring $k[[T^2, T^3]]\cong k[[x,y]]/(x^2 - y^3)$, the supremum is $1$.
  \item For the semigroup ring $k[[T^2, T^5]]$, the supremum is $2=\eh(R)$ and it is achieved by $I_2$ and $I_4$.
  \item For the semigroup ring $k[[T^3, T^5]]$, a direct check of all possible combinations shows that the supremum is $2$, which is strictly less than the multiplicity of $k[[T^3, T^5]]$.
\end{enumerate}
\end{example}

In contrast to Proposition~\ref{prop: sup in dim one}, Example~\ref{example: sup of integrally closed dim one}, and Example~\ref{example: sup of integrally closed numerical semigroups}, we show that the supremum is often $\infty$ when $R$ has dimension at least two.

\begin{proposition}
\label{prop: sup in dim at least two}
Let $(R,\m)$ be a Noetherian local ring of dimension $d\geq 2$. Then we have
$$\sup_{\substack{I\subsetneq J,  \sqrt{I}=\m}} \left\{\frac{\eh (I) - \eh (J)}{\length (R/I) - \length (R/J)} \right\} = \infty.$$
If, in addition, $R/\m$ is algebraically closed, then
$$\sup_{\substack{I\subsetneq J, \sqrt{I}=\m \\ I=\overline{I}, J=\overline{J}}} \left\{\frac{\eh (I) - \eh (J)}{\length (R/I) - \length (R/J)} \right\} = \infty.$$
\end{proposition}
\begin{proof}
We may assume that $R=\widehat{R}$ is complete. Let $P$ be a minimal prime of $R$ such that $\dim(R/P)=d$. We note that
$$\sup_{\substack{I\subsetneq J,  \sqrt{I}=\m}} \left\{\frac{\eh (I) - \eh (J)}{\length (R/I) - \length (R/J)} \right\} \geq \sup_{\substack{P\subseteq I\subsetneq J,  \sqrt{I}=\m}} \left\{\frac{\eh (I) - \eh (J)}{\length (R/I) - \length (R/J)} \right\}$$
simply because the second supremum is taken over a smaller set. Next, for any $P\subseteq I\subsetneq J$ we have
$$\length (R/I) - \length (R/J) = \length\big((R/P)/I(R/P)\big) - \length\big((R/P)/J(R/P)\big)$$
while by the additivity formula for multiplicities,
\[
  \eh (I) - \eh (J)  = \sum_{\mf p\in\Minh(R)}\length(R_{\mf p})\big(\eh (I, R/\mf p) - \eh (J, R/\mf p)\big)   \geq  \eh(I, R/P) - \eh(J, R/P).
\]
It follows that
$$\sup_{\substack{P\subseteq I\subsetneq J,  \sqrt{I}=\m}} \left\{\frac{\eh (I) - \eh (J)}{\length (R/I) - \length (R/J)} \right\} \geq \sup_{\substack{P\subseteq I\subsetneq J,  \sqrt{I}=\m}} \left\{\frac{\eh (I, R/P) - \eh (J, R/P)}{ \length\big((R/P)/I(R/P)\big) - \length\big((R/P)/J(R/P)\big)} \right\}.$$

Therefore to prove the supremum in question is $\infty$, we may replace $R$ by $R/P$ to assume that $(R,\m)$ is a Noetherian complete local domain. By Cohen's structure theorem, we can find a complete regular local ring $(A,\m_A)$ with $A/\m_A\cong R/\m$ such that $A\to R$ is a module-finite extension. Take an element $x$ in a minimal generating set of $\m_A$ and let $\mf a=\m_A^n$ and $\mf b=\m_A^n+(x^{n-1})$. One checks directly that $\eh(\mf a)=n^d$ and $\eh(\mf b)=n^d-n^{d-1}$. Let $I=\mf a R$ and $J=\mf b R$ be their expansions to $R$. Then we have $\eh(I)-\eh(J)=n^{d-1}\rank_A(R)$. Tensoring the short exact sequence $$0\to A/\m_A \to A/\mf a \to A/\mf b \to 0$$ with $R$ we find that $\length(R/I)-\length(R/J) \leq \length(R/\m_A R)$. It follows that
$$\frac{\eh (I) - \eh (J)}{\length (R/I) - \length (R/J)} \geq \frac{n^{d-1}\rank_A(R)}{\length(R/\m_A R)}.$$
Letting $n\to\infty$ shows that the supremum is $\infty$, which proves of the first statement.

For the second statement, by exactly the same argument as above, we may assume that $(R,\m)$ is a Noetherian complete local domain (using the fact that the pre-image of integrally closed ideals in $R/P$ are integrally closed in $R$). Now by Proposition~\ref{prop: improved Watanabe} we can find $I:=\overline{\m^n} \subseteq J$ such that $\length(R/I)-\length(R/J)=1$ and $J$ is integrally closed. Then $J=I+(x)$ for some $x\in (I:\m)-I$. For a general element $x'=ux+v \in \m^n+(x)$, where $u\in R^\times$ and $v\in \m^n$, we have $\m^n+(x)= \m^n+(x')$ and thus $$\eh(\m^n+(x), R)= \eh(\m^n+(x), R/(x'))=\eh(\m^n+(x'), R/(x'))=\eh(\m^n, R/(x')).$$
It follows that
$$\eh(I)-\eh(J)= \eh(\m^n) - \eh(\m^n+(x))= \eh(\m^n) - \eh(\m^n, R/(x'))=n^{d-1}(n\eh(R)- \eh(R/(x'))).$$
Since $R$ is a complete local domain, $I$ is integrally closed and $I\subsetneq J$, we know that $\eh(I)> \eh(J)$ and thus $n\eh(R)-\eh(R/(x'))>0$. But then $n\eh(R)-\eh(R/(x'))\geq 1$ since it is an integer and thus $\eh(I)-\eh(J)= n^{d-1}(n\eh(R)- \eh(R/(x'))) \geq n^{d-1}$. Letting $n\to\infty$ completes the proof of the second statement. 
\end{proof}

\begin{remark}
We suspect that $$\sup_{\substack{I\subsetneq J, \sqrt{I}=\m \\ I=\overline{I}, J=\overline{J}}} \left\{\frac{\eh (I) - \eh (J)}{\length (R/I) - \length (R/J)} \right\} = \infty$$ should always be true when $\dim(R)\geq 2$ without any assumption on the residue field.
\end{remark}

\section{Hilbert--Kunz multiplicity and colength}

\subsection{Infima} In this subsection, we study the infima of the ratio between Hilbert--Kunz multiplicity and colength, and of the relative drops of Hilbert--Kunz multiplicities. We first show that for integrally closed ideals, these infima are often $1$.
\begin{proposition}
\label{prop: inf integrally closed Hilbert-Kunz}
Let $(R, \mf m)$ be a Noetherian local ring of prime characteristic $p > 0$ and dimension $d\geq 1$. Consider the following two infima:
$$\inf_{\substack{\sqrt{I} = \mf m , I = \overline{I}}} \left \{ \frac{\ehk(I)}{\length(R/I)}\right\} \quad \text{ and } \quad
\inf_{\substack{I \subsetneq J, \sqrt{I} = \mf m \\ I = \overline{I}, J=\overline{J}}} \left \{\frac{\ehk (I) - \ehk (J)}{\length (R/I) - \length (R/J)}  \right\}.$$
Then we have
\begin{enumerate}
  \item $\widehat{R}$ is not equidimensional then both infima are equal to $0$.
  \item If $\widehat{R}$ is equidimensional and $R/\m$ is perfect, then both infima are $\geq 1$.
  \item If there exists a minimal prime $P$ of $\widehat{R}$ such that $\dim(\widehat{R}/P)=d$ and $\widehat{R}_P$ is a field, then both infima are $\leq 1$.
\end{enumerate}
\end{proposition}
\begin{proof}
First of all, $(2)$ follows directly from Corollary~\ref{cor: relative HK inequality}. For $(1)$, the proof is nearly identical to the proof of Proposition~\ref{prop: non equidimensional}: there we have constructed a sequence of $\m$-primary integrally closed ideals $\{I_n\}_n$ such that $\length(R/I_n)\to\infty$ but $\ehk(I_n)\leq \eh(I_n)\leq C$. It follows immediately from this that both infimum are $0$.  As for $(3)$, by Lemma \ref{lem: main approx}, there exists a constant $C>0$ and a sequence of
integrally closed $\m$-primary ideals $\{J_n\}_n$ such that $$\ehk(J_n) \leq \eh (J_n)\leq (1 + C/n) \length (R/J_n)$$
and $\length (R/J_n)\to\infty$. It follows immediately that both infima are $\leq 1$ (see the proof of Theorem~\ref{thm: inf of relative drop of integrally closed}).
\end{proof}

\begin{remark}
If $\widehat{R}$ is unmixed and $R/\m$ is perfect, the first infimum above achieves $1$ if and only if $R$ is regular due to Corollary~\ref{cor: HK integral closed and singular}. On the other hand, \cite[Theorem~1.8, Proposition~1.9]{WatanabeYoshidaMcKay} show that the second infimum above often achieves $1$ in a two-dimensional rational singularity.
\end{remark}

In contrast, the infima for all ideals or all tightly closed ideals seem much more subtle.

\begin{remark}
\label{rmk: inf of eHK}
We believe that both
\[
\rho(R):= \inf_{\sqrt{I} = \mf m } \left \{ \frac{\ehk(I)}{\length (R/I)} \right \} \,\  \text{ and } \,\
\rho^*(R) :=  \inf_{\sqrt{I} = \mf m , I = I^*} \left \{ \frac{\ehk(I)}{\length (R/I)} \right \}
\]
are interesting new invariants of $R$. We make some first observations on them.
\begin{enumerate}
  \item It follows from Proposition~\ref{prop: inf integrally closed Hilbert-Kunz} that $\rho(R)=\rho^*(R)=0$ when $\widehat{R}$ is not equidimensional.
  \item We clearly have $\rho(R)\leq \rho^*(R)$, and in general they could be different even for hypersurfaces, see Example~\ref{example: inf different from inf*}.
  \item Since $\eh(I)/d!\leq \ehk(I)\leq \eh(I)$, we have $1/d!n(R) \leq \rho(R) \leq 1/n(R)$.
  \item We have $\rho(R)=1$ when $R$ is a complete intersection \cite[Theorem 2.1]{MillerFrobeniusCI}, in particular, $\rho(R)$ could be different from $1/d!n(R)$ (since it is well-known and easy to check that $n(R)=1$ if and only if $R$ is Cohen--Macaulay \cite{StuckradVogel,KMQSY}).
  \item Both $\rho(R)$ and $\rho^*(R)$ could be $<1$ when $R$ is Cohen--Macaulay (in fact, even when $R$ is strongly F-regular), see Example~\ref{example: SFR not Roberts}. In particular, $\rho(R)$ could be different from $1/n(R)$.
\end{enumerate}
\end{remark}

\begin{example}
\label{example: inf different from inf*}
Let $R = k[[x,y]]/(x^2)$. Since $\dim(R)=1$, $\ehk(I) = \eh (I)$ and tightly closed ideals are exactly integrally closed ideals by Lemma~\ref{lem: tc and ic in dimension one}, thus
$\rho^* (R) > 1$ by Theorem~\ref{thm: inf of integrally closed}.
On the other hand, since $R$ is a hypersurface, we know that $\rho(R)=1$ by Remark~\ref{rmk: inf of eHK} (4).
More generally, this argument shows that if $S$ is a one-dimensional complete intersection then $\rho (S) = \rho^* (S)$ if and only if
the regular locus of $\widehat{S}$ is not empty.
\end{example}

\begin{example}
\label{example: SFR not Roberts}
Let $X=(x_{ij})$ be an $m\times n$ matrix of indeterminants where $m\geq n\geq 2$ and $m>2$. Let $S=k[[X]]$ be a power series ring in these $m\times n$ variables and let $I_2(X)\subseteq S$ be the ideal generated by $2\times 2$ minors of $X$. Set $R=S/I_2(X)$, then it is well-known that $R$ is strongly F-regular (see \cite[Example 7.14]{HochsterHunekeJAG}) and thus $\rho(R)=\rho^*(R)$ since every ideal is tightly closed. By \cite[Example 7.11]{KuranoNumericalEquivalence}, $R$ is not numerically Roberts, i.e., there exists an $\m$-primary ideal $I$ of finite projective dimension such that $\ehk(I)\neq \length(R/I)$ by \cite[Theorem 6.4]{KuranoNumericalEquivalence}. We will show below that this formally implies the slightly stronger fact that there exists an $\m$-primary ideal $I$ of finite projective dimension such that $\ehk(I)<\length(R/I)$, which implies that $\rho(R)=\rho^*(R)<1$ as wanted.

To see this stronger fact, we suppose on the contrary that $\ehk(I)\geq \length(R/I)$ for every $\m$-primary ideal $I$ of finite projective dimension. Now for every module $M$ of finite length and finite projective dimension, by \cite[the paragraph after Lemma 4.1]{SmokePerfectModules}, we have
$$0\to M \to R/J \to N\to 0$$
where $\pd_RR/J<\infty$ and $N$ has a filtration with quotients of the form $R/(x_1,\dots,x_d)$ where $x_1,\dots,x_d$ is some system of parameters of $R$. Base change along the Frobenius, computing length, taking limit, and using that $\Tor_1(F^e_*R, N)=0$ by \cite[Th\'{e}or\`{e}me (I.7)]{PeskineSzpiro} we obtain
$$\lim_{e\to\infty}\frac{\length(F^e(M))}{p^{ed}} + \lim_{e\to\infty}\frac{\length(F^e(N))}{p^{ed}} = \ehk(J),$$
where the two limits above exist due to \cite{SeibertComplexesFrobenius}. We next note that $\lim_{e\to\infty}\frac{\length(F^e(N))}{p^{ed}}=\length(N)$ since $N$ has a filtration by $R/(x_1,\dots,x_d)$. Comparing the above equality with $\length(M)+\length(N)=\length(R/J)$ and noting that $\ehk(J)\geq \length(R/J)$ by our assumption, we find that
\begin{equation}
\label{equation: length module finite projective dimension}
\lim_{e\to\infty}\frac{\length(F^e(M))}{p^{ed}} \geq \length(M) \text{ for every $M$ of finite length and finite projective dimension.}
\end{equation}
On the other hand, we also have a short exact sequence
$$0\to M'\to (R/(y_1,\dots,y_d))^{\oplus n}\to M\to 0$$
where $n=\nu(M)$ and $y_1,\dots, y_d$ is a system of parameter of $R$ that is contained in $\Ann(M)$. Since $\pd_RM<\infty$, we know that $\pd_RM'<\infty$. Thus after base change along the Frobenius, computing length, and taking limit we obtain that
$$\lim_{e\to\infty}\frac{\length(F^e(M))}{p^{ed}}+ \lim_{e\to\infty}\frac{\length(F^e(M'))}{p^{ed}} = n\ehk((y_1,\dots, y_d))=n\length(R/(y_1,\dots, y_d))=\length(M)+\length(M').$$
This combined with (\ref{equation: length module finite projective dimension}) shows that we must have equality in (\ref{equation: length module finite projective dimension}). In particular, we have $\ehk(I)=\length(R/I)$ for every $\m$-primary ideal $I$ of finite projective dimension, which is a contradiction (a similar argument shows that there also exists an $\m$-primary ideal $I$ of finite projective dimension such that $\ehk(I)>\length(R/I)$).
\end{example}

It is well-known that for a Noetherian local ring $(R,\m)$ of prime characteristic $p>0$. The infimum of the relative drop of Hilbert--Kunz multiplicity is equal to the F-signature \cite{WatanabeYoshidaMin,PolstraTucker}.
\[
\inf_{\substack{I\subsetneq J, \sqrt{I}=\m }} \left \{
\frac{\ehk (I) - \ehk (J)}{\length (R/I) - \length (R/J)}
\right \} =\fsig(R).
\]
Here the F-signature $s(R)$ is an invariant that measures the singularities of strongly F-regular rings: we have $0\leq s(R)\leq 1$ and $s(R)>0$ if and only if $R$ is strongly F-regular \cite{AberbachLeuschke}. It seems natural to investigate the infimum of the relative drop of Hilbert--Kunz multiplicity for tightly closed ideals. In this direction we propose the following question:

\begin{question}
\label{question: infimum positive tc}
Let $(R,\m)$ be a Noetherian local ring of prime characteristic $p>0$. Then is
\[
\inf_{\substack{I\subsetneq J, \sqrt{I}=\m \\ I=I^*, J=J^*}} \left \{
\frac{\ehk (I) - \ehk (J)}{\length (R/I) - \length (R/J)}
\right \} > 0?
\]
\end{question}

Of course, Question \ref{question: infimum positive tc}  holds when $R$ is strongly F-regular by the discussion above. But note that proving this infimum is positive for all \emph{weakly} F-regular rings would imply the long standing conjecture on the equivalence of weak and strong F-regularity.

We may reword the question as the existence of $\delta > 0$
such that whenever the relative Hilbert--Kunz multiplicity is not $0$, then
it is at least $\delta$. Such behavior is known to hold in the case where we have a finite scheme parametrizing the pairs of ideals by the techniques of \cite{SmirnovTucker}.
Let us present two such cases now which will serve as an evidence.
Essentially these cases result from fixing
the smaller ideal $I$ in Question~\ref{question: infimum positive tc}.

\begin{proposition}
\label{prop: go socle}
Let $(R,\m)$ be a Noetherian local ring of prime characteristic $p>0$. For a fixed $\mf m$-primary tightly closed ideal $I$  we have
\[
\inf_{I\subsetneq J=J^*} \left \{
\frac{\ehk (I) - \ehk (J)}{\length (R/I) - \length (R/J)}
\right \}
= \inf_{\substack{I\subsetneq J=J^* \\ \m J\subseteq I}} \left \{
\frac{\ehk (I) - \ehk (J)}{\length (R/I) - \length (R/J)}\right \}
\]
\end{proposition}
\begin{proof}
First, let us note that
\[
\inf_{I\subsetneq J=J^*} \left \{
\frac{\ehk (I) - \ehk (J)}{\length (R/I) - \length (R/J)}
\right \}
=
\inf_{I\subsetneq J} \left \{
\frac{\ehk (I) - \ehk (J)}{\length (R/I) - \length (R/J)}
\right \}.
\]
One inequality is due to the second infimum being taken over a larger set,
while the other inequality holds after observing that
\begin{equation}\label{equation: relative and tc}
\frac{\ehk (I) - \ehk (J)}{\length (R/I) - \length (R/J)}
= \frac{\ehk (I) - \ehk (J^*)}{\length (R/I) - \length (R/J)}
\geq \frac{\ehk (I) - \ehk (J^*)}{\length (R/I) - \length (R/J^*)}.
\end{equation}

Next, using \cite[Proposition~3.5]{SmirnovTucker} (which is valid for any ideal), we obtain that for any ideal $J$ there exists a socle ideal $J' \subseteq I :_R \m$ such that
\[
\frac{\ehk (I) - \ehk (J)}{\length (R/I) - \length (R/J)} \geq
\frac{\ehk (I) - \ehk (J')}{\length (R/I) - \length (R/J')}.
\]
Finally, we use (\ref{equation: relative and tc}) again and observe
that $(J')^* \subseteq I :_R \m$. Namely, let $x \in (J')^*$, so there exists $c$ not in any minimal prime of $R$ such that $c x^q \in (J')^{[q]}$ for $q \gg 0$. It follows that $c x^{q} \m^{[q]} \subseteq \m^{[q]}(J')^{[q]} \subseteq I^{[q]}$,
hence $\m x \in I^* = I$.
\end{proof}

The next proposition essentially follows directly from \cite[Theorem 3.1]{HochsterYao}.

\begin{proposition}
\label{prop: inf for fixed tc I}
Let $(R,\m)$ be an excellent, reduced and equidimensional local ring of prime characteristic $p>0$. For a fixed $\mf m$-primary tightly closed ideal $I$ we have
\[
\inf_{I\subsetneq J} \left \{
\frac{\ehk (I) - \ehk (J)}{\length (R/I) - \length (R/J)}
\right \} > 0.
\]
\end{proposition}
\begin{proof}
We first note that
\[
\frac{\ehk (I) - \ehk (J)}{\length (R/I) - \length (R/J)} \geq \frac{\ehk (I) - \ehk (J)}{\length (R/I)}.
\]
Therefore
\begin{align*}
\inf_{I\subsetneq J} \left \{
\frac{\ehk (I) - \ehk (J)}{\length (R/I) - \length (R/J)}
\right \} &\geq
\frac{1}{\length (R/I)}
\inf_{I\subsetneq J} \left \{
\ehk (I) - \ehk (J)
\right \} 
\end{align*}
Since $R$ is excellent, reduced and equidimensional, $R$ has a completely stable test element (see \cite[Theorem 6.20]{HochsterHuneke1}) and $\widehat{R}$ is reduced and equidimensional. It follows from \cite[Theorem 8.17]{HochsterHuneke1} that $\ehk(I)\neq \ehk(J)$ for all $J\supsetneq I$ since $I$ is tightly closed. Thus the infimum above is positive by \cite[Theorem~3.1]{HochsterYao}.
\end{proof}

\begin{remark}
A different proof can be obtained by using the Grassmannian to parametrize
the socle ideals, see \cite[Corollary~3.7]{SmirnovTucker} and \cite[Remark~4.17]{SmirnovAffine}.
\end{remark}

In \cite{HochsterYao, SmirnovTucker} it was proved that if $I$ is a parameter ideal
then the infimum, in fact, does not depend on $I$.
We want to extend this result to $I$ being the tight closure of a parameter ideal.
For this, we need to
generalize the idea of transferring relative Hilbert--Kunz multiplicity to local cohomology
which originates from \cite[Proposition~2.3, Theorem~2.4]{HochsterYao}.

%

\begin{proposition}
\label{prop: go local cohomology}
Let $(R, \mf m)$ be an excellent, reduced and equidimensional local ring of prime
characteristic $p > 0$ and dimension $d$.
Denote $H = \lc^d_\mf m (R)$ and $T = 0^*_H$.
Then we have the following correspondence between certain $\m$-primary ideals of $R$ and submodules of $H$.
\begin{enumerate}
\item For every system of parameters $\ul x=x_1,\dots, x_d$ of $R$ and an ideal $I\supseteq (\ul x)^*$,
there exists a submodule $L \supseteq T$ of $H$ such that
$I/(\ul x)^* \cong L/T$.
\item Conversely, given $T \subseteq L \subset H$
such that $\length (L/T) < \infty$, there exists a system of parameters
$\ul x=x_1,\dots,x_d$ and an ideal $I\supseteq (\ul x)^* $, such that $I/(\ul {x})^* \cong L/T$.
\item Moreover, there is a system of parameters $\ul{x}=x_1,\dots,x_d$ such that for any $L$ such that $T\subseteq L \subseteq T :_H \m$, there exists an ideal $I\supseteq (\ul x)^* $, such that $I/(\ul {x})^* \cong L/T$.
\end{enumerate}

Furthermore, this correspondence extends to Hilbert--Kunz multiplicity. Namely,
the natural inclusion $i_L\colon I/(\ul x)^*\cong L/T \to H/T$ induces
maps $i_{L,e} \colon L/T \otimes_R F_*^e R \to H/T \otimes_R F_*^e R$ and we have
\[
\lim_{e \to \infty} \frac{\length_{F_*^e R} \left (\im i_{L,e} \right )}{p^{ed}} = \ehk ((\ul x)) - \ehk(I).
\]
\end{proposition}
\begin{proof}
First, it is well-known (for example, by the argument as in \cite[Proposition 3.3]{Smith}) that
\begin{equation}
\label{equation: direct limit lc}
\varinjlim_{n} R/(\ul x^n)^* \cong \lc^d_\mf m (R) / 0^*_{\lc^d_\mf m (R)} = H/T
\end{equation}
and that all maps in the direct limit system are injective: for if $\overline{r}\in R/(\ul x^n)^*$ is mapped to $0$ in the direct limit system, then there exists $t\geq 0$ such that $r(x_1\cdots x_d)^t\in (\ul x^{n+t})^*$ and it follows from colon-capturing that $r\in (\ul x^n)^*$ and thus $\overline{r}=0$. Here we can use colon-capturing since $R$ is excellent and equidimensional (see \cite[Theorem 2.3]{Huneke}). $(1)$ and $(2)$ follow immediately from this.

To see (3), we note that by applying $\Hom_R(R/\m, -)$ to (\ref{equation: direct limit lc}) we obtain that
$$\varinjlim_{n} \frac{(\ul x^n)^*:_R \m}{(\ul x^n)^*} \cong \frac{T:_H \m}{T}$$
where all maps in the direct limit system remain injective (as $\Hom_R(R/\m, -)$ is left exact). Since $(T:_H \m)/T$ is a finite dimensional $(R/\m)$-vector space, there exists $n\gg0$ such that $((\ul x^n)^*:_R \m)/(\ul x^n)^*\cong (T:_H \m)/T$. Now $(3)$ is clear by replacing $(\ul x)$ by $(\ul x^n)$.

We now prove the last statement. We consider the following commutative diagram
\[\xymatrix{
I^{[p^e]}/((\ul x)^*)^{[p^e]} \ar[r]^-{i_{L,e}} \ar@{^{(}->}[d] & L^{[p^e]}/T^{[p^e]} \ar@{^{(}->}[d] \\
R/((\ul x)^*)^{[p^e]} \ar[r] \ar@{->>}[d] & H/T^{[p^e]} \ar@{->>}[d] \\
R/(\ul x^{p^e})^* \ar@{^{(}->}[r] & H/T
}.
\]
Here we note that $L^{[p^e]}/T^{[p^e]}\cong \im i_{L,e}$ after identifying $F^e_*R$ with $R$, and we have a natural surjection $H/T^{[p^e]}\twoheadrightarrow H/T$ since $T\subseteq H$ is an F-stable submodule of $H$ (under the natural Frobenius action on $H=\lc^d_\mf m (R)$) and $T^{[p^e]}$ can be identified with the $R$-span of the image of $T$ under the $e$-th iterates of this Frobenius action, in particular $T^{[p^e]}\subseteq T$. Chasing the image of $I^{[p^e]}/((\ul x)^*)^{[p^e]}$ in the diagram above we find that
$$
\length(I^{[p^e]}/((\ul x)^*)^{[p^e]}) \geq \length_{F^e_*R}(\im i_{L,e}) = \length(L^{[p^e]}/T^{[p^e]})\geq \length((L^{[p^e]}+T)/T) \geq \length((I^{[p^e]}+ (\ul x^{p^e})^*)/ (\ul x^{p^e})^*)
$$
where the last inequality follows from the injectivity of the last row of the diagram. Thus we have
\begin{align*}
\ehk((\ul x)) - \ehk(I) & = \ehk((\ul x)^*) - \ehk(I) \\
   & = \lim_{e\to\infty}\frac{\length(I^{[p^e]}/((\ul x)^*)^{[p^e]})}{p^{ed}} \\
   & \geq \lim_{e\to\infty} \frac{\length_{F^e_*R}(\im i_{L,e})}{p^{ed}}  \\
   & \geq \lim_{e\to\infty} \frac{\length((I^{[p^e]}+ (\ul x^{p^e})^*)/ (\ul x^{p^e})^*)}{p^{ed}}\\
   & \geq  \lim_{e\to\infty} \frac{\length(R/ (\ul x^{p^e})^*)}{p^{ed}} - \lim_{e\to\infty} \frac{\length(R/I^{[p^e]})}{p^{ed}} \\
   & = \ehk((\ul x)) - \ehk(I)
\end{align*}
where the last equality follows from \cite[Lemma 1.13]{SmirnovEquimultiplicity} (note that $R$ has a test element since $R$ is excellent and reduced, see \cite[Theorem 6.20]{HochsterHuneke1}, so the assumptions of \cite[Lemma 1.13]{SmirnovEquimultiplicity} are satisfied). It follows that we must have equality throughout and the result follows.
\end{proof}

Now we can prove our partial result towards Question~\ref{question: infimum positive tc}, this result can be viewed as a generalization of the main result of \cite{HochsterYao} (there the main result is the case that $(\ul x)^*=(\ul x)$, i.e., $R$ is F-rational).

\begin{corollary}
\label{cor: tight inf relative eHK sop}
Let $(R,\m)$ be an excellent, reduced and equidimensional local ring of prime characteristic $p>0$. Then we have
\[
\inf_{I=(\ul x)^*\subsetneq J=J^*} \left \{
\frac{\ehk (I) - \ehk (J)}{\length (R/I) - \length (R/J)}
\right \}
> 0,
\]
where $I=(\ul x)^*\subsetneq J=J^*$ means that $I$ is the tight closure of some system of parameters $(\ul x)$ of $R$ (i.e., we are allowed to vary the system of parameters).
\end{corollary}
\begin{proof}
First of all, by Proposition~\ref{prop: go socle}, we know that
$$\inf_{I=(\ul x)^*\subsetneq J=J^*} \left \{
\frac{\ehk (I) - \ehk (J)}{\length (R/I) - \length (R/J)}
\right \} = \inf_{\substack{I=(\ul x)^*\subsetneq J=J^* \\ \m J\subseteq I}} \left \{
\frac{\ehk (I) - \ehk (J)}{\length (R/I) - \length (R/J)}
\right \}$$
Then by the last assertion of Proposition~\ref{prop: go local cohomology} (we use the same notation as in Proposition~\ref{prop: go local cohomology}), we have that
$$\inf_{\substack{I=(\ul x)^*\subsetneq J=J^* \\ \m J\subseteq I}} \left \{
\frac{\ehk (I) - \ehk (J)}{\length (R/I) - \length (R/J)}
\right \} \geq \inf_{\substack{T \subsetneq L \\ \m L\subseteq T}}
\lim_{e \to \infty} \frac{\length_{F_*^e R} \left (\im i_{L, e} \right )}{p^{ed}\length (L/T)}.$$
By Proposition~\ref{prop: go local cohomology} (3), there exists a single system of parameters $\underline{y}=y_1, \ldots, y_d$ of $R$ such that
\[
\inf_{\substack{T \subsetneq L \\ \m L\subseteq T}}
\lim_{e \to \infty} \frac{\length_{F_*^e R} \left (\im i_{L, e} \right )}{p^{ed}\length (L/T)}
= \inf_{\substack{(\underline{y})^*\subsetneq J \\ \m J \subseteq(\ul y)^*}} \left \{
\frac{\ehk ((\ul y)^*) - \ehk (J)}{\length (R/((\ul y)^*) - \length (R/J)} \right \}.
\]
But now since $\underline{y}$ is fixed, the latter infimum is positive by Proposition~\ref{prop: inf for fixed tc I}.
\end{proof}

\subsection{Suprema} Finally, in this subsection we study the supremum of the relative drops of Hilbert--Kunz multiplicities. We first recall the following simple but critical lemma. Among other things, this lemma implies that the supremum is always $\leq \ehk(R)$, and that if we consider the supremum of the ratio between Hilbert--Kunz multiplicity and colength, then we always get $\ehk(R)$ (and the supremum is achieved at the maximal ideal).

\begin{lemma}[{\cite[Lemma~4.2]{WatanabeYoshida}}]
\label{lem: sup HK upper}
Let $(R, \mf m)$ be a Noetherian local ring of prime characteristic $p > 0$. Then for every pair of $\m$-primary ideals $I\subsetneq J$, we have
\[
\ehk(I)\leq \ehk(R)\length(R/I) \,\ \text{ and } \,\ \frac{\ehk (I) - \ehk (J)}{\length (R/I) - \length (R/J)} \leq \ehk(R).
\]
\end{lemma}

Our first observation is that using parameter ideals we can always reach the upper bound (with more assumptions this can be also deduced from \cite[Theorem~3.5]{EpsteinValidashti}).

\begin{proposition}
\label{prop: sup relative eHK}
Let $(R, \mf m)$ be a Noetherian local ring of prime characteristic $p > 0$
and dimension $d \geq 1$ and let $J = (x_1, \ldots, x_d)$ be a parameter ideal.
Then $$\ehk (\mf mJ) - \ehk(J) = d \ehk(R).$$
In particular, we have
\[
\sup_{I\subsetneq J, \sqrt{I}=\m} \left \{
\frac{\ehk (I) - \ehk (J)}{\length (R/I) - \length (R/J)} \right \} = \ehk (R).
\]
\end{proposition}
\begin{proof}
Consider the following short exact sequence
$$0 \to J/\mf m J \to R/\mf mJ \to R/J \to 0.$$
Since $J / \mf m J \cong (R/\mf m)^{\oplus d}$, base change along the Frobenius we obtain that
\[
\length (R/\frq{(\mf mJ)}) + \length_{F_*^e R} \left(\Tor_1^R (R/J, F_*^e R) \right) \geq
d \length (R/\frq{\mf m}) + \length (R/\frq{J}).
\]
Note that there is a natural surjection $H_1(x_1,\dots, x_d; F^e_*R) \twoheadrightarrow \Tor_1^R (R/J, F_*^e R)$, and thus by \cite[Theorem 7.3.5]{RobertsBook}
we deduce that $\length_{F_*^e R} \left(\Tor_1^R (R/J, F_*^e R) \right )  = o (p^{ed})$.
It follows that
$$\ehk (\mf mJ) - \ehk(J) \geq d \ehk(R).$$
But then we must have equality above by Lemma~\ref{lem: sup HK upper} as $d = \length (J/\mf mJ)$. The second assertion clearly follows from the first assertion and Lemma~\ref{lem: sup HK upper}.
\end{proof}

The result above has a consequence in tight closure theory, which we believe should be known to experts, but we could not find a reference at this level of generality. If $R$ is excellent and analytically irreducible, it follows from
\cite[Proposition~4.2]{VraciuSpTC} as $(\m J)^*$ is contained in the \emph{special}
tight closure.

\begin{corollary}
\label{cor: tight closure containment}
Let $(R, \mf m)$ be a Noetherian local ring of prime characteristic $p > 0$
and dimension $d \geq 1$ and let $J$ be a parameter ideal.
Then $(\mf m J)^* \cap J = \mf m J$. In particular, if $R$ is F-rational, then
$\mf mJ$ is tightly closed.
\end{corollary}
\begin{proof}
Suppose that $\mf m J \subsetneq (\mf m J)^* \cap J$. Since we know that
$\ehk (I) = \ehk (I^*)$, by Proposition~\ref{prop: sup relative eHK} we have that
\[
\ehk (R) = \frac{\ehk (\mf mJ) - \ehk (J)}{\length (R/\mf mJ) - \length (R/J)}
< \frac{\ehk ((\mf mJ)^* \cap J) - \ehk (J)}
{\length (R/ (\mf mJ)^* \cap J) - \length(R/J)},
\]
where the strict inequality is because the numerator is the same (and nonzero by Proposition~\ref{prop: sup relative eHK}) and the denominator strictly decrease. But this contradicts Lemma~\ref{lem: sup HK upper}. For the last assertion, note that if $R$ is F-rational then
$(\m J)^*\subseteq J^*=J$ and thus
\[(\mf m J)^*= (\mf m J)^* \cap J = \mf m J. \qedhere\]
\end{proof}


As a consequence, we obtain the following partial result on the supremum of relative drop of Hilbert--Kunz multiplicity for tightly closed ideals.

\begin{corollary}
\label{cor: tight sup relative eHK}
Let $(R, \mf m)$ be a Noetherian local ring of prime characteristic $p > 0$
and dimension $d \geq 1$ and let $J$ be a parameter ideal.
Suppose that $J^* = J + (\mf m J)^*$ (e.g., $R$ is F-rational),
then
\[
\sup_{\substack{I\subsetneq J, \sqrt{I}=\m \\ I=I^*, J=J^*}} \left \{
\frac{\ehk (I) - \ehk (J)}{\length (R/I) - \length (R/J)} \right \} = \ehk (R).
\]
\end{corollary}
\begin{proof}
By Corollary~\ref{cor: tight closure containment} we have $(\mf m J)^* \cap J = \mf m J$, it follows that
\[
\frac{J^*}{(\mf m J)^*} =
\frac{J + (\mf m J)^*}{(\mf m J)^*} \cong \frac{J}{J \cap (\mf m J)^*}
= \frac{J}{\mf m J}.
\]
Since $\ehk(I) = \ehk(I^*)$, the conclusion follows from Lemma~\ref{lem: sup HK upper} and Proposition~\ref{prop: sup relative eHK}.
\end{proof}

\begin{remark}
\label{rmk: sup eHK tc and ic}
We caution the readers that it is not always true that
\[
\sup_{\substack{I\subsetneq J, \sqrt{I}=\m \\ I=I^*, J=J^*}} \left \{
\frac{\ehk (I) - \ehk (J)}{\length (R/I) - \length (R/J)} \right \} = \ehk (R).
\]
In fact, our Example~\ref{example: sup of integrally closed dim one} and Example~\ref{example: sup of integrally closed numerical semigroups} give examples of one-dimensional Noetherian local rings such that the supremum above differs from $\ehk(R)$. To see this, note that in dimension one, $\ehk(I)=\eh(I)$ and tightly closed ideals are exactly integrally closed ideals by Lemma~\ref{lem: tc and ic in dimension one}. It follows that in these examples,
\[
\sup_{\substack{I\subsetneq J, \sqrt{I}=\m \\ I=I^*, J=J^*}} \left \{
\frac{\ehk (I) - \ehk (J)}{\length (R/I) - \length (R/J)} \right \} = \sup_{\substack{I\subsetneq J, \sqrt{I}=\m \\ I=\overline{I}, J=\overline{I}}} \left \{
\frac{\ehk (I) - \ehk (J)}{\length (R/I) - \length (R/J)} \right \} = \sup_{\substack{I\subsetneq J, \sqrt{I}=\m \\ I=\overline{I}, J=\overline{I}}} \left \{
\frac{\eh (I) - \eh (J)}{\length (R/I) - \length (R/J)} \right \}.
\]
But in Example~\ref{example: sup of integrally closed dim one} and Example~\ref{example: sup of integrally closed numerical semigroups} we have seen that the last supremum above could be different from $\ehk(R)=\eh(R)$.
\end{remark}


\bibliographystyle{plain}
\bibliography{refs}

\end{document}